\documentclass[12pt,a4paper]{article}
\usepackage{authblk}
\usepackage[margin=2cm]{geometry}
\usepackage{t1enc}
\usepackage[utf8]{inputenc}
\usepackage{amsthm,amsmath,amssymb}
\usepackage{graphicx}
\usepackage{enumerate}
\usepackage{hyperref}
\usepackage{bm}
\usepackage{comment}
\usepackage{amsfonts}
\usepackage{graphicx,caption}
\usepackage{bm}
\usepackage{amsmath, amsthm, amssymb}
\usepackage{graphicx}
\usepackage{hyperref}
 \usepackage{relsize}
\usepackage{algpseudocode}

\usepackage{bbm}

\theoremstyle{plain}
\usepackage{amsthm}
\makeatletter
\newcommand{\newreptheorem}[2]{\newtheorem*{rep@#1}{\rep@title}\newenvironment{rep#1}[1]{\def\rep@title{#2 \ref*{##1}}\begin{rep@#1}}{\end{rep@#1}}}
\makeatother

\newtheorem{theorem}{Theorem}
\newtheorem*{theorem-non}{Theorem}
\newtheorem*{non-lemma}{Lemma} 
\newtheorem{lemma}[theorem]{Lemma}
\newreptheorem{lemma}{Lemma}

\newtheorem{conjecture}[theorem]{Conjecture}
\theoremstyle{definition}
\newtheorem{remark}[theorem]{Remark}

\DeclareMathOperator{\Sub}{Sub}
\DeclareMathOperator{\Tr}{Tr}

\DeclareMathOperator{\RowS}{RowSpace}

\DeclareMathOperator{\cok}{cok}

\DeclareMathOperator{\Sur}{Sur}
\DeclareMathOperator{\Hom}{Hom}
\DeclareMathOperator{\Aut}{Aut}

\DeclareMathOperator{\dkl}{D_{KL}}

\begin{document}
\title{Cohen-Lenstra distribution for sparse matrices with determinantal biasing}
\author{Andr\'as M\'esz\'aros}
\maketitle
\begin{abstract}
Let us consider the following matrix $B_n$. The columns of $B_n$ are indexed with $[n]=\{1,2,\dots,n\}$ and the rows are indexed with $[n]^3$. The row corresponding to $(x_1,x_2,x_3)\in [n]^3$ is given by $\sum_{i=1}^3 e_{x_i}$, where $e_1,e_2,\dots,e_n$ is the standard basis of $\mathbb{R}^{[n]}$. Let $A_n$ be random $n\times n$ submatrix of $B_n$, where the probability that we choose a  submatrix $C$  is proportional to $|\det(C)|^2$.  
   
    Let $p\ge 5$ be a prime. We prove that the asymptotic distribution of the $p$-Sylow subgroup of the cokernel of $A_n$ is given by the Cohen-Lenstra heuristics.

    Our result is motivated by the conjecture that the first homology group of a random two dimensional hypertree is also Cohen-Lenstra distributed. 
\end{abstract}

\section{Introduction}

Let us consider the following matrix $B_n$. The columns of $B_n$ are indexed with $[n]=\{1,2,\dots,n\}$ and the rows are indexed with $[n]^3$. The row corresponding to $(x_1,x_2,x_3)\in [n]^3$ is given by $\sum_{i=1}^3 e_{x_i}$, where $e_1,e_2,\dots,e_n$ is the standard basis of $\mathbb{R}^{[n]}$. We define a random $n$-element subset $X_n$ of $[n]^3$ as follows. For any deterministic $n$-element subset $K$ of $[n]^3$, we set
\[\mathbb{P}(X_n=K)=\frac{|\det(B_n[K])|^2}{\det(B_n^T B_n)}.\]
Here $B_n[K]$ is the submatrix of $B_n$ induced by the set of rows $K$. It follows from the Cauchy-Binet formula that this indeed gives us a probability measure. Note that the law of $X_n$ is a so called determinantal probability measure \cite{lyons2003determinantal,hough2006determinantal}.

The random $n\times n$ matrix $A_n$ is defined as $A_n=B_n[X_n]$.


The subgroup of $\mathbb{Z}^{[n]}$ generated by the rows of $A_n$ is denoted by $\RowS(A_n)$. The cokernel $\cok(A_n)$ of $A_n$ is defined as the factor group $\mathbb{Z}^{[n]}/\RowS(A_n)$. Since $\det(A_n)\neq 0$ by the definition of $A_n$, $\cok(A_n)$ is a finite abelian group of order $|\det(A_n)|$.

\begin{theorem}\label{thm1}
Let $5\le p_1<p_2<\dots<p_s$ be primes. Let $\Gamma_{n,i}$ be the $p_i$-Sylow subgroup of $\cok(A_n)$. Let $G_i$ be a finite abelian $p_i$-group. Then
\[\lim_{n\to\infty} \mathbb{P}\left(\bigoplus_{i=1}^s \Gamma_{n,i}\simeq \bigoplus_{i=1}^s G_i\right)=\prod_{i=1}^s\left(|\Aut(G_i)|^{-1} \prod_{j=1}^\infty (1-p_i^{-j}) \right).\]
\end{theorem}

The distribution appearing in Theorem~\ref{thm1} is the one that appears in the Cohen-Lenstra
heuristics. It was introduced by Cohen and Lenstra \cite{cohen2006heuristics} in a conjecture on the
distribution of class groups of quadratic number fields.  Since then, these distributions and their modified versions also appeared in the setting of random matrices and random graphs, see  \cite{friedman1987distribution,wood2019random,wood2023probability,clancy2015cohen,wood2017distribution} for some classical results, and \cite{recent1,recent2,recent3,recent4,recent5,recent6,recent7} for more recent results. Note that all the results above are for dense matrices, while our random matrix $A_n$ is sparse, each row contains at most $3$ non-zero entries. The only other known examples of such sparse matrices with Cohen-Lenstra behaviour  are the reduced Laplacian of a random $d$-regular graph \cite{meszaros2020distribution}, and the adjacency matrix of a random $d$-regular graph provided that $d$ and $p$ are coprimes \cite{nguyen2018cokernels}. Considering somewhat denser matrices, there are examples of  $n\times n$  random matrices with Cohen-Lenstra limiting behaviour and $O(n\log n)$ non-zero entries \cite{meszaros2024phase,kang2024random}, or with $O(n^{1+\varepsilon})$ non-zero entries \cite{nguyen2022random}.

We prove Theorem~\ref{thm1} by calculating the moments of $\cok(A_n)$.

\begin{theorem}\label{thm2}
    Let $G$ be a finite abelian group such that $|G|$ is coprime to $6$. Then
    \[\lim_{n\to\infty} \mathbb{E}|\Sur(\cok(A_n),G)|=1.\]
\end{theorem}

It is well known that Theorem~\ref{thm2} implies Theorem~\ref{thm1} by the results of Wood \cite{wood2017distribution} on the moment problem. For $s=1$, we can also use the earlier results of \cite{ellenberg2016homological}.

Note that the condition that $|G|$ is coprime to $6$ can not be omitted, see Remark~\ref{remark1} and Remark~\ref{remark2}.

\subsection{A motivation: Two dimensional hypertrees}

Let $I_n$ be an ${{n-1}\choose 2}\times {{n}\choose 3}$ matrix, where the rows are index with ${[n-1]\choose 2}$, that is, the set of two element subsets of $[n-1]$ and the columns are indexed with ${{[n]}\choose 3}$, that is, the set of three element subsets of $[n]$. Moreover, given a two element subset $A$ of $[n-1]$ and a three element subset $B=\{x_1<x_2<x_3\}$ of $[n]$, we have
\[I_n(A,B)=\begin{cases}
    (-1)^i&\text{if }A=B\backslash\{x_i\},\\
    0&\text{otherwise.}
\end{cases}\]

Let $C_n$ be a random ${{n-1}\choose 2}$ element subset of  ${{[n]}\choose 3}$ such that for any ${{n-1}\choose 2}$ element subset $K$ of  ${{[n]}\choose 3}$, we have
\begin{equation}\label{hypdef}
\mathbb{P}(C_n=K)=\frac{|\det(I_n^T[K])|^2}{\det(I_nI_n^T)}.
\end{equation}

If we add all the subsets of $[n]$ of size $1$ and $2$ to $C_n$, we get a random two dimensional simplicial complex. With a slight abuse of notation, we also use $C_n$ to denote this complex. \break $C_n$ is called a random $2$-dimensional hypertree, and it is the two dimensional analogue of a uniform random spanning tree of the complete graph.  The simplicial complexes which are taken with positive probability are called  $2$-dimensional hypertrees (on the vertex set $[n]$).  It turns out that $\cok(I_n^T[C_n])$ has a topological meaning. It is the first homology group $H_1(C_n)$ of~$C_n$.   Generalising Cayley's formula for the number of spanning trees, Kalai \cite{kalai1983enumeration} proved that
\[\sum_{K} |H_1(K)|^2=n^{{n-2}\choose {2}},\]
where the summation is over the  $2$-dimensional hypertrees on the vertex set $[n]$. Thus, \eqref{hypdef} can be rewritten as
\[\mathbb{P}(C_n=K)=\frac{|H_1(K)|^2}{n^{{n-2}\choose {2}}}.\]

See \cite{kahle2022topology,meszaros2022local,werf2022determinantal,meszaros2023coboundary,meszaros2024bounds,meszaros20242} for results on these random hypertrees. Determinantal random complexes were also investigated in \cite{lyons2009random}.

In \cite{kahle2022topology} the following conjecture was stated:

\begin{conjecture}\label{conj1}
    The $p$-Sylow subgroup of $H_1(C_n)$ also follows the Cohen-Lenstra heuristics. The same should be true if we consider the uniform measure on the set of hypertrees.
\end{conjecture}
For computational evidence of the second conjecture, see \cite{kahle2020cohen}. Note that by now we know that the first part of Conjecture \ref{conj1} is false in the special case of $p=2$~\cite{meszaros20242}. This is nicely aligned with the fact that in Theorem~\ref{thm1} we also need to assume that $p_1\neq 2$.

The matrices $B_n$ and $I_n^T$ has similar structure, but $I_n^T$ is less symmetric which makes things harder. Our results demonstrate that sparse random matrices obtained from determinantal measures can have cokernels that follow the Cohen-Lenstra heuristics. This should make the first part of Conjecture \ref{conj1} more plausible for $p>2$.

\subsection{Outline of the proof}

Note that $\mathbb{E}|\Sur(\cok(A_n),G)|$ is the expected number of vectors $q\in G^n$ such that $A_nq=0$ and the components of $q$ generate $G$. To prove Theorem~\ref{thm2}, we determine the limit of the expected number of such vectors in three steps:
\begin{itemize}
    \item Step 1: We can restrict our attention to vectors $q$ such that the distribution of a uniformly chosen component of $q$ is close to the uniform distribution on some subgroup $H$ of $G$. The contribution of all the other vectors is negligible, see 
    Lemma~\ref{lemmafontos}.
    \item Step 2: It is enough to consider vectors $q$ such that the above mentioned subgroup $H$ is equal to $G$. In other words, we can restrict our attention to vectors $q$ such that the distribution of a uniformly chosen component of $q$ is close to the uniform distribution on~$G$. Let us call these vectors equidistributed. The contribution of the non-equidistributed vectors is negligible. See equation~\eqref{HneG}.
    \item Step 3: Finally, we determine the contribution from the equidistributed vectors. See equation~\eqref{HeG}
\end{itemize}

This three step strategy is quite standard in the literature. For matrices with independent entries these steps can be carried out using discrete Fourier transform, see for example \cite{wood2019random}. However, the entries of $A_n$ are not independent. Thus, in our case we need new and significantly different methods. Namely, the proof of steps 1 and 2 relies on techniques from information theory. More specifically, the notion of Kullback–Leibler divergence and Pinsker's inequality. Step 3 relies on the multivariable Laplace's method.

\section{Acknowledgments}
 The author is grateful to  B\'alint Vir\'ag for his comments. The author thanks the anonymous reviewers for their valuable suggestions.  The author was supported by the NSERC discovery grant of B\'alint Vir\'ag and the KKP 139502 project.

\section{Preliminaries}

\subsection{Kullback–Leibler divergence}

Let $\mathcal{X}$ be a finite set and let $\nu$ and $\mu$ be two probability measures on $\mathcal{X}$. The Kullback–Leibler divergence of $\nu$ and $\mu$ is defined as
\[\dkl(\nu||\mu)=\sum_{x\in \mathcal{X}}\nu(x)\log\left(\frac{\nu(x)}{\mu(x)}\right).\]
(If $\nu(x)=0$, then $\nu(x)\log\left(\frac{\nu(x)}{\mu(x)}\right)$ is defined to be $0$. If $\nu(x)\neq 0$ and $\mu(x)=0$ for some $x$, then $\dkl(\nu||\mu)$ is defined to be $+\infty$. Here we use the natural logarithm, but many authors prefer to use the base $2$ logarithm.)

For the proof of the next lemma see \cite[Lemma 2.30]{alajaji2018introduction}.
\begin{lemma}[Gibbs' inequality]\label{Gibbs} $\dkl(\nu||\mu)\ge 0$.
\end{lemma}

The next lemma is an easy consequence of the fact that refinement cannot decrease divergence \cite[Lemma 2.33]{alajaji2018introduction}.
\begin{lemma}\label{refineq}
Let $\mathcal{Y}\subset \mathcal{X}$, then
\[\dkl(\nu||\mu)\ge \nu(\mathcal{Y})\log \frac{\nu(\mathcal{Y})}{\mu(\mathcal{Y})}+(1-\nu(\mathcal{Y}))\log \frac{1-\nu(\mathcal{Y})}{1-\mu(\mathcal{Y})}.\]
\end{lemma}

Let 
\[\|\nu-\mu\|_1=\sum_{x\in \mathcal{X}}|\nu(x)-\mu(x)|.\]
For the proof of the next lemma see \cite[Lemma 2.37]{alajaji2018introduction}.\footnote{Note that in \cite{alajaji2018introduction} a different constant is given, because they use the binary logarithm instead of the natural one.}
\begin{lemma}[Pinsker's inequality]\label{Pinsker}
We have
\[\|\nu-\mu\|_1\le \sqrt{2 \dkl(\nu||\mu)}.\]
\end{lemma}

\subsection{Expressing the surjective moments}

For any $n\times n$ matrix $C$ with integer entries and a finite abelian group $G$, we have
\[|\Sur(\cok(C),G)|=|\{q\in G^n\,:\,Cq=0\text{ and the components of $q$ generate }G\}|,\]
see for example \cite[Proposition 5.2.]{meszaros2020distribution}.

Thus,
\begin{equation}\label{eqSur}
\mathbb{E}    |\Sur(\cok(A_n),G)|=\sum_q \mathbb{P}(A_n q=0),
\end{equation}
where the summation is over all $q\in G^n$ such that  the components of $q$ generate $G$.

\begin{remark}\label{remark1}
    Let $G=\mathbb{Z}/3\mathbb{Z}$. For $i=1,2$, consider $q_i=(i+3\mathbb{Z},i+3\mathbb{Z},\dots,i+3\mathbb{Z})\in G^n$. Since the sum of the entries is $3$ in each row of $A_n$, we see that $\mathbb{P}(A_nq_1=0\text{ and }A_nq_2=0)=1$. This means that the $3$-Sylow subgroup of $\cok(A_n)$ is never trivial, which excludes that the $3$-Sylow subgroup of $\cok(A_n)$ has Cohen-Lenstra limiting distribution. Also, $\mathbb{E}    |\Sur(\cok(A_n),G)|\ge 2$, so the conclusion of Theorem~\ref{thm2} does not hold for $G=\mathbb{Z}/3\mathbb{Z}$.
\end{remark}
\section{Formulas for fixed $n$}

Let $G$ be a finite abelian group, $q\in G^n$. For $a\in G$, let
\begin{align}n_a&=n_a(q)=|\{i\in [n]\,:\,q_i=a\}|,\\
m_a&=m_a(q)=|\{(i,j)\in [n]^2\,:\,q_i+q_j+a=0\}=\sum_{b\in G} n_bn_{-a-b},\label{madef}\\
G_+&=G_+(q)=\{a\in G\,:\,n_a>0\}\label{Gpdef}.
\end{align}

Let $M=M_q$ be a matrix where the rows and columns are indexed with $G_+$ and the entries are given by
\begin{equation}
    M(a,b)=\begin{cases}
2n_an_{-2a}+m_{a}&\text{for $a=b$,}\\
2\sqrt{n_an_b}\,n_{-a-b}&\text{for $a\neq b$.}
\end{cases}\label{Mdef}
\end{equation}

\begin{lemma}\label{lemmafinite}
We have
\[\mathbb{P}(A_nq=0)=\frac{1}3 n^{-2n}\det(M)\prod_{a\in G_+} m_a^{n_a-1}.\]

Moreover, the matrix $M$ is positive semi-definite.
\end{lemma}
\begin{proof}
Let 
\[I=\{(x_1,x_2,x_3)\in [n]^3\,:\, q_{x_1}+q_{x_2}+q_{x_3}=0\},\]
and let $B_{n,q}$ be the submatrix of $B_n$ consisting of the rows with index in $I$. It follows from the Cauchy- Binet formula that 
\begin{equation}\label{BqpB}\mathbb{P}(A_nq=0)=\mathbb{P}(X_n\subset I)=\sum_{\substack{K\subset I\\|K|=n}}\frac{|\det(B_n[K])|^2}{\det(B_n^TB_n)}=\frac{\det(B_{n,q}^TB_{n,q})}{\det(B_n^TB_n)}.\end{equation}

Note that
\[B_{n,q}^TB_{n,q}=\sum_{k=1}^3\sum_{\ell=1}^3\sum_{(x_1,x_2,x_3)\in I} e_{x_k}^Te_{x_\ell}.\]
Let $(k,\ell)\in [3]^2$ such that $k\neq \ell$, and let $h$ be the unique element of $[3]\backslash \{k,\ell\}$. For $(i,j)\in [n]^2$, we have
\begin{align*}\left(\sum_{(x_1,x_2,x_3)\in I} e_{x_k}^Te_{x_\ell}\right)(i,j)&=|\{(x_1,x_2,x_3)\in I\,:\, x_k=i,x_\ell=j\}|\\&=|\{x_h\in [n]\,:\,q_i+q_j+q_{x_h}=0\}|=n_{-q_i-q_j}.
\end{align*}
Also, if $i\neq j$, then clearly
\[\left(\sum_{(x_1,x_2,x_3)\in I} e_{x_h}^Te_{x_h}\right)(i,j)=0,\]
and
\begin{align*}\left(\sum_{(x_1,x_2,x_3)\in I} e_{x_h}^Te_{x_h}\right)(i,i)&=|\{(x_1,x_2,x_3)\in I\,:\, x_h=i\}|\\&=|\{(x_k,x_\ell)\in [n]^2\,:\,q_i+q_{x_k}+q_{x_\ell}=0\}|=m_{q_i}.
\end{align*}

Therefore,
\[B_{n,q}^TB_{n,q}(i,j)=\begin{cases}
6n_{-2q_i}+3m_{q_i}&\text{for $i=j$},\\
6n_{-q_i-q_j}&\text{for $i\neq j$}.
\end{cases}\]

For $a\in G_+$, let us define the vector $w_a\in \mathbb{R}^{[n]}$, by
\[w_a(i)=\begin{cases}
\frac{1}{\sqrt{n_a}}&\text{if $q_i=a$,}\\
0&\text{otherwise,}
\end{cases}\]
and let us define the subspace
\[W_a=\{v\in \mathbb{R}^{[n]}\,:\,v(i)=0\text{ for all $i$ such that $q_i\neq a$, and }\langle v,w_a\rangle =0\}.\]

It is straightforward to check that $W_a$ is an invariant subspace of $B_{n,q}^TB_{n,q}$, and restricted to this subspace $B_{n,q}^TB_{n,q}$ is just $3m_{a}$ times the identity. Also note that $\dim W_a=n_a-1$.

The subspace $U$ spanned by $(w_a)_{a\in G_+}$ is also an invariant subspace of $B_{n,q}^TB_{n,q}$. A straightforward calculation gives that the matrix of the restriction  of $B_{n,q}^TB_{n,q}$ to the subspace $U$ in the basis $(w_a)_{a\in G_+}$ is $3M$, where $M$ is the matrix defined above.

For all $a\in G_+$ choose a basis of $W_a$. Together with $(w_a)_{a\in G_+}$ they form a basis of $\mathbb{R}^{[n]}$.  Writing $B_{n,q}^TB_{n,q}$ in this basis, we see that
\begin{equation}\label{eqBq}\det(B_{n,q}^TB_{n,q})=\det(3M)\prod_{a\in G_+} (3m_a)^{n_a-1}=3^n\det(M)\prod_{a\in G_+} m_a^{n_a-1}.
\end{equation}
Choosing $q=0$, the matrix $M$ is just the $1\times 1$ matrix $(3n^2)$, furthermore, $G_+=\{0\}$, $m_0=n^2$, $n_0=n$ and $m_a=n_a=0$ for $a\neq 0$. Thus, \eqref{eqBq} gives us
\begin{equation}\label{eqB}
 \det(B_{n}^TB_{n})=\det(B_{n,0}^TB_{n,0})=3^{n+1}n^{2n}.
\end{equation}

Inserting \eqref{eqBq} and \eqref{eqB} into \eqref{BqpB}, the first statement of the lemma follows.

To see that $M$ is positive semidefinite, observe that $3M$ is obtained from the positive semidefinite operator $B_{n,q}^TB_{n,q}$ by restricting it to an invariant subspace, so $M$ is also positive semidefinite. 
\end{proof}

Now let $\underline{n}=(n_a)_{a\in G}$ be a vector of non-negative integers summing up to $n$. We define $G_+$ and $m_a$ as above, that is,
\begin{equation}\label{defmaGp}
m_a=\sum_{b\in G} n_bn_{-a-b}\quad\text{ and }\quad G_+=\{a\in G\,:\,n_a>0\}.
\end{equation}
We also define the matrix $M$ as in \eqref{Mdef}. 
Let 
\[S=S(\underline{n})=\{q\in G^n\,:\,n_a(q)=n_a\text{ for all } a\in G\},\]
and
\[E=E(\underline{n})=\sum_{q\in S} \mathbb{P}(A_nq=0).\]

Using Lemma \ref{lemmafinite}, we see that
\begin{equation}\label{elsoE}E=\frac{|S|}3 n^{-2n}\det(M)\prod_{a\in G_+} m_a^{n_a-1}=\frac{n!}{\prod_{g\in G} n_a!}\frac{n^{-2n}\det(M)}3\prod_{a\in G_+} m_a^{n_a-1}.
\end{equation}

\begin{lemma}\label{E0}
Assume that there is an $a\in G_+$ such that $m_a=0$, then $E=0$. 
\end{lemma}
\begin{proof}
Let $a\in G_+$ such that $m_a=0$. Recall that $m_a=\sum_{b\in G} n_bn_{-a-b}$. Thus, if $m_a=0$, then $n_{-a-b}$ must be $0$ for all $b\in G_+$. Then  all the entries of $M$ in the column indexed with $a$ are $0$. Thus, $\det(M)=0$ and the statement follows from \eqref{elsoE}. 
\end{proof}

We define the probability measures $\nu=\nu_{\underline{n}}$ and $\mu=\mu_{\underline{n}}$ on $G$ by setting
\[\nu(a)=\frac{n_a}{n}\quad\text{ and }\quad\mu(a)=\frac{m_a}{n^2}\quad\text{ for all }a\in G.\]

    Assume that $m_a>0$ for all $a\in G_+$, then \eqref{elsoE} can be rewritten as
    \[E=\frac{n!}{\prod_{a\in G} n_a!} \frac{\det(M)}{3\prod_{a\in G_+} m_a}\prod_{a\in G_+} \mu(a)^{n\nu(a)}. \]
    We can further rewrite this in terms of the Kullback–Leibler divergence as follows
    \begin{equation}\label{eqEealpha}E=\alpha(\underline{n})\frac{\det(M)}{3\prod_{a\in G_+} m_a}\exp\left(-n\dkl(\nu||\mu)\right), \end{equation}
    where
    \[\alpha(\underline{n})=\frac{n!}{\prod_{a\in G} n_a!}\Bigg/\exp\left(n\sum_{a\in G}-\nu(a)\log\nu(a)\right).\]

    For all $\underline{n}$, we have $\alpha(\underline{n})\le 1$. See for example \cite[Lemma 2.2]{csiszar2004information}.

    Note that 
    \[\Tr M= \sum_{a\in G_+}\left(2n_a n_{-2a}+m_a\right)\le 2n\sum_{a\in G} n_a+\sum_{a\in G} m_a=3n^2.\]Combining this with the fact that $M$ is positive semidefinite, we obtain that 
    \[\det(M)\le (\Tr M)^{|G_+|}\le c_G n^{2|G|},\]
    for  $c_G=3^{|G|}$. Thus,
    \begin{equation}\label{Ebecs}
        E\le c_G n^{2|G|}\exp(-n\dkl(\nu||\mu)).
    \end{equation}

\section{Estimating the  Kullback–Leibler divergence}

Let $\nu$ a probability distribution on $G$, let $\mu$ be the distribution of $-(X_1+X_2)$, where $X_1$ and $X_2$ are independent random variables with distribution $\nu$. The next lemma and its proof are similar to \cite[Lemma 4.3]{meszaros2020distribution}.

\begin{lemma}\label{tvvsdkl}
Assume that $|G|$ is not divisible by $3$. There is a constant $C$ depending on $G$ such that for any $\nu$ and $\mu$ like above, there is a subgroup $H$ of $G$ such that for any $a\in G$
\[|\nu_H(a)-\nu(a)|\le C\sqrt{\dkl(\nu||\mu)},\]
moreover, $\nu(G\backslash H)\le C\dkl(\nu||\mu)$.

Here $\nu_H$ is the uniform distribution on $H$.

\end{lemma}
 \begin{proof}

 Let $\delta=\|\nu-\mu\|_1$. Pinsker's inequality (Lemma \ref{Pinsker}) gives us that $\delta\le \sqrt{2\dkl(\nu||\mu)}$. By choosing the constant $C$ large enough, we may assume that $\delta$ is sufficiently small.
 
 We use discrete Fourier transform, that is, for a character $\varrho\in \hat{G}=\Hom(G,\mathbb{C}^*)$, we define
 \[\hat{\nu}(\varrho)=\sum_{a\in G}\varrho(a)\nu(a)\quad\text{ and }\quad\hat{\mu}(\varrho)=\sum_{a\in G}\varrho(a)\mu(a).\]

Our conditions imply that $\hat{\mu}(\varrho)=\overline{(\hat{\nu}(\varrho))^2}$ and $|\hat{\nu}(\varrho)-\hat{\mu}(\varrho)|\le \delta$. Thus, $z=\hat{\nu}(\varrho)$ satisfies the inequality, $|z-\overline{z^2}|\le \delta$. So let us investigate the function $f(z)=z-\overline{z^2}$. First, we determine the roots of $f$. If $z$ is a root of $f$, then $|z|=|z|^2$, so $|z|$ is either $0$ or $1$. If $|z|=1$ and $z$ is a root, then we have the equation $\arg z\equiv -2\arg z \mod{2\pi}$. Thus, it follows that the four roots of $f$ are $0,1$ and $-\frac{1}2\pm\frac{\sqrt{3}}2 i$. Consider $f$ as a real two variable function, that is, consider $g(a,b)=(\Re f(a+bi),\Im f(a+bi))=(a-a^2+b^2,b+2ab)$. The Jacobian matrix of this map is\[\begin{pmatrix}
    1-2a&2b\\
    2b&1+2a
\end{pmatrix},\] 
which has determinant $1-4(a^2+b^2)$. Thus, one can check that the Jacobian matrix is nonsingular at the roots of $f$. Thus, there are $c$ and $\delta_0>0$ with the property that if $|f(z)|<\delta_0$, then there is a root $z_0$ of $f$ such that $|z-z_0|\le c|f(z)|$. In particular, assuming that $\delta$ is small enough, for any character $\varrho$, we have $|\hat{\nu}(\varrho)-z_0|\le c\delta$ for some root $z_0$ of $f$. Note that since $3$ does not divide $|G|$, the characters of $G$ do not take the values of $-\frac{1}2\pm\frac{\sqrt{3}}2i$. Thus, it follows that $\left|\hat{\nu}(\varrho)-\left(\frac{1}2\pm\frac{\sqrt{3}}2i\right)\right|>\varepsilon$ for some positive $\varepsilon$. Thus, assuming that $\delta$ is small enough, for any character $\varrho\in\hat{G}$, we have either $|\hat{\nu}(\varrho)|\le c\delta$ or $|\hat{\nu}(\varrho)-1|\le c\delta$. 

Let $\hat{G}_1$ be the set of characters $\varrho$ such that $|\hat{\nu}(\varrho)-1|\le c\delta$. Let
\[H=\cap_{\varrho\in \hat{G}_1} \ker \varrho.\]

There exists $\alpha>0$ such that for all $\varrho\in \hat{G}$ and $a\in G\backslash \ker \varrho$, we have $\Re \varrho(a)\le 1-\alpha$. Thus, for $\varrho\in\hat{G}_1$, on the one hand, we have
\[\Re \hat{\nu}(\varrho)\ge 1-|1-\hat{\nu}(\varrho)|\ge 1-c\delta,\]
and on the other hand
\[\Re \hat{\nu}(\varrho)\le 1-\alpha \nu(G\backslash \ker \varrho).\]
Combining these observations, $\nu(G\backslash \ker \varrho)\le \frac{c\delta}{\alpha}$. Thus, it follows that $\nu(G\backslash H)\le  \frac{c\delta|G|}{\alpha}$.

It is clear that $\hat{\nu_H}(\varrho)=1$ for all $\varrho\in \hat{G}_1$. 

Take any $\varrho\in \hat{G}\backslash \hat{G}_1$. We claim that assuming that $\delta$ is small enough $H$ is not a subset of $\ker \varrho$. Indeed, if $H$ was a subset of $\ker \varrho$, then this would imply that
\[|\hat{\nu}(\varrho)|\ge \nu(H)-\nu(G\backslash H)=1-2\nu(G\backslash H)\ge 1- \frac{2c\delta|G|}{\alpha}>c\delta  \]
provided that $\delta$ is small enough, which is a contradiction. Thus, there is an $h\in H$ such that $\varrho(h)\neq 1$. Note that $h+H=H$, so
\[\hat{\nu_H}(\varrho)=\frac{1}{|H|}\sum_{a\in H}\varrho(a)=\frac{1}{|H|}\sum_{a\in H}\varrho(a+h)=\frac{1}{|H|}\varrho(h)\sum_{a\in H}\varrho(a)=\varrho(h)\hat{\nu_H}(\varrho).\]

Since $\rho(h)\neq 1$, it follows that $\hat{\nu_H}(\varrho)=0$. 

Thus, $|\hat{\nu}(\varrho)-\hat{\nu_H}(\varrho)|\le c\delta$ for all $\varrho\in\hat{G}$.

Using the inverse Fourier transform
\[|\nu(a)-\nu_H(a)|=\left|\frac{1}{|G|}\sum_{\varrho\in \hat{G}}\overline{\varrho(a)}(\hat{\nu}(\varrho)-\hat{\nu_H}(\varrho))\right|\le \frac{1}{|G|}\sum_{\varrho\in \hat{G}}|\hat{\nu}(\varrho)-\hat{\nu_H}(\varrho)|\le c\delta\le c \sqrt{2\dkl(\nu||\mu)}.  \]

We have got the first statement of the lemma. To prove the second statement, let 
\[p=\nu( G\backslash H)\quad\text{ and }\quad q=\mu( G\backslash H).\]
Lemma \ref{refineq} gives us that
\begin{equation}\label{eqdklin}\dkl(\nu||\mu)\ge p\log\frac{p}q+(1-p)\log\frac{1-p}{1-q}.\end{equation}

We claim that $q\ge 2p(1-p)$. Indeed, let $X_1$ and $X_2$ be two independent random variables with law $\nu$. Then \[q=\mathbb{P}(-X_1-X_2\notin H)\ge \mathbb{P}(X_1\in H,X_2\notin H)+\mathbb{P}(X_1\notin H,X_2\in H)=2p(1-p).\]

Using the already proven part of the lemma, assuming that $\delta$ is small enough, we may assume that $p$ is small enough. Let us assume that $p\le \frac{1}2$. For $p\le \frac{1}2$, we have
$p\le 2p(1-p)\le q.$ For a fixed $p$, the right hand side of \eqref{eqdklin} as a function of $q$ is monotone increasing on $[p,1]$. Thus, replacing $q$ with $2p(1-p)$, we still get a valid inequality, so
\[\dkl(\nu||\mu)\ge p\log\frac{p}{2p(1-p)}+(1-p)\log\frac{1-p}{1-2p(1-p)}.\]

A simple calculation shows that the right hand side is $0$ at $p=0$ and it has a positive derivative at $p=0$. Thus, there are two positive constants  $c_2$ and $p_0$ such that for all $0\le p\le p_0$, we have
\[\dkl(\nu||\mu)\ge p\log\frac{p}{2p(1-p)}+(1-p)\log\frac{1-p}{1-2p(1-p)}\ge c_2p=c_2\nu( G\backslash H).\]
The statement follows.
\end{proof}

\section{Asymptotic estimates -- The proof of Theorem~\ref{thm2}}

Let 
\[D_n=\left\{\underline{n}\in \mathbb{Z}_{\ge 0}^G\,:\,\sum_{a\in G} n_a=n \text{ and the support of $\underline{n}$ generates $G$}\right\},\]
and
\[D_n'=\{\underline{n}\in D_n\,:\, m_a>0 \text{ for all }a\in G_+\},\]
where $m_a$ and $G_+$ were defined in \eqref{defmaGp}.

Combining \eqref{eqSur} with the definition of $E(\underline{n})$ and $D_n$, we see that
\begin{equation}\label{SurSE}
    \mathbb{E}|\Sur(\cok(A_n),G)|=\sum_{\underline{n}\in D_n} E(\underline{n}).
\end{equation}

Let  \[t_n=2C\sqrt{|G|n\log n}\qquad\text{ and }\qquad r_n=4C|G|\log n,\]
where $C$ is the constant provided by Lemma \ref{tvvsdkl}. 

Let $H$ be a subgroup of $G$, we define $u_H\in \mathbb{R}^G$ as
\[u_H(g)=\begin{cases}
    \frac{1}{|H|}&\text{for $g\in H$},\\
    0&\text{for $g\notin H$}.
\end{cases}\]
For a positive integer $n$, let
\[B(n,H)=\left\{\underline{n}\in D_n'\,:\, \|\underline{n}-nu_H\|_\infty\le t_n\text{ and }\sum_{a\notin H}n_a\le r_n\right\}.\]

Let $\Sub(G)$  be the set of the subgroups of $G$.

\begin{lemma}\label{lemmafontos}
We have
\[\lim_{n\to \infty} \left(\sum_{\underline{n}\in D_n} E(\underline{n})-\sum_{H\in\Sub(G)} \sum_{\underline{n}\in B(n,H)} E(\underline{n})\right)=0. \]

\end{lemma}

\begin{proof}
    It follows from Lemma \ref{E0} that
    \[\sum_{\underline{n}\in D_n} E(\underline{n})=\sum_{\underline{n}\in D_n'} E(\underline{n}).\]
    Using Lemma \ref{tvvsdkl} we see that if $\dkl(\nu_{\underline{n}}||\mu_{\underline{n}})\le \frac{4|G|\log n}{n}$ for some $\underline{n}\in D_n'$, then
    $\underline{n}\in B(n,H)$ for some subgroup $H$ of $G$. Let $D_n''=D_n'\backslash \cup_{H\in\Sub(G)} B(n,H)$. We just proved that if $\underline{n}\in D_n''$, then $\dkl(\nu_{\underline{n}}||\mu_{\underline{n}})>\frac{4|G|\log n}{n}$. Combining this with \eqref{Ebecs}, we see that for any $\underline{n}\in D_n'',$, we have
    \[E(\underline{n})\le c_Gn^{2|G|}\exp(-4|G|\log n)=c_G n^{-2|G|} .\]

    Clearly $|D_n''|\le (n+1)^{|G|}$.
    Note that for large enough $n$ the sets $B(n,H)$ are pairwise disjoint. So
    \[0\le \sum_{\underline{n}\in D_n} E(\underline{n})-\sum_{H\in\Sub(G)} \sum_{\underline{n}\in B(n,H)} E(\underline{n})=\sum_{\underline{n}\in D_n''} E(\underline{n})\le |D_n''|c_G n^{-2|G|}\le c_G\left(\frac{n+1}{n^2}\right)^{|G|}.\]
    Thus, the statement clearly follows.
\end{proof}

In Section \ref{HneGsec}, we prove that for any subgroup $H$ of $G$ such that $H\neq G$, we have
\begin{equation}\label{HneG}
 \lim_{n\to\infty} \sum_{\underline{n}\in B(n,H)} E(\underline{n})=0.   
\end{equation}

In Section \ref{secLaplace}, we prove that
\begin{equation}\label{HeG}
 \lim_{n\to\infty} \sum_{\underline{n}\in B(n,G)} E(\underline{n})=1.   
\end{equation}

Combining Lemma \ref{lemmafontos} with the equations \eqref{SurSE}, \eqref{HneG} and \eqref{HeG}, we obtain Theorem~\ref{thm2}.
\subsection{The proof of equation \eqref{HneG}}\label{HneGsec}
Let $H$ be a subgroup of $G$ such that $H\neq G$. Recalling that $|G|$ is odd, let $|G/H|=2t+1$, and choose $g_1,g_2,\dots,g_t$ such that $H,g_1+H,-g_1+H,\dots,g_t+H,-g_t+H$ is a list of all the cosets of $H$. For $1\le i\le t$, let $F_i=(g_i+H)\cup (-g_i+H)$, and let $F_0=H$.

Take an $\underline{n}\in B(n,H)$. As before, let $G_+$ be the support of $\underline{n}$. By reindexing if necessary, we may assume that for some $h$, $G_+$ intersects  $F_1,F_2,\dots,F_h$ and does not intersect $F_i$ for $i>h$. Assuming that $n$ is large enough, we have $H=F_0\subset G_+$. Let $k$ be the number of $1\le i\le h$ such that $G_+$ intersect both $g_i+H$ and $-g_i+H$.

\begin{lemma}\label{lemmaMbecs}
Given this $\underline{n}$, let $M$ be the matrix defined in \eqref{Mdef}. Then
\[\frac{\det(M)}{\prod_{a\in G_+} m_a}=O\left(r_n^{|G|-|H|}\left(\frac{t_n}{n}\right)^k\right).\]
Throughout the paper the hidden constant in the $O$-notation may depend on $G$.
\end{lemma}
\begin{proof}
Let $M_i$ be the submatrix of $M$ determined by the rows and columns indexed by $F_i\cap G_+$. Since $M$ is positive semi-definite, by Fisher's inequality (\cite[Theorem 7.8.5]{horn2012matrix}), we have
\begin{equation*}
    \det(M)\le \prod_{i=0}^h \det(M_i).
\end{equation*}

Thus,
\begin{equation}\label{eqFisher}\frac{\det(M)}{\prod_{a\in G_+} m_a}\le \prod_{i=0}^{h}\frac{\det(M_i)}{\prod_{a\in F_i\cap G_+} m_a}.\end{equation}






For any $0\le i\le h$, and $a\in F_i\cap G_+$, we have
\[M_i(a,a)=2n_a n_{-2a}+m_a\le 3m_a.\]

Using the Hadamard's inequality (\cite[Theorem 7.8.1]{horn2012matrix}), we get that

\begin{equation}\label{detMi1}
    \frac{\det(M_i)}{\prod_{a\in F_i\cap G_+} m_a}\le 3^{|F_i\cap G_+|}=O(1).
\end{equation}

Now assume that $1\le i\le h$, and $(g_i+H)\cap G_+$ and $(-g_i+H)\cap G_+$ are both nonempty.  Under these assumptions we can prove a bound stronger than \eqref{detMi1}.

Let
\[s_1=\sum_{a\in g_i+H} n_a\quad \text{ and }\quad s_2=\sum_{a\in -g_i+H} n_a.\]

For $a\in (g_i+H)\cap G_+$, we have
\[
m_a=\sum_{b\in G} n_b n_{-a-b}=\sum_{b\in H} n_b n_{-a-b}+\sum_{b\in -a+H} n_b n_{-a-b}+\sum_{b\in G\backslash(H\cup (-a+H))} n_b n_{-a-b}
.\]

Since $n_b=\frac{n}{|H|}+O(t_n)$ for all $b\in H$, the first sum on the right hand side is 
\[\sum_{b\in H} n_b n_{-a-b}=\sum_{b\in H} \left(\frac{n}{|H|}+O(t_n)\right) n_{-a-b}=\left(\frac{n}{|H|}+O(t_n)\right)\sum_{c\in -g_i+H} n_c=\left(\frac{n}{|H|}+O(t_n)\right)s_2.\]

The second sum is
\[\sum_{b\in -g_i+H} n_b n_{-a-b}=\sum_{b\in -g_i+H} n_b \left(\frac{n}{|H|}+O(t_n)\right)=\left(\frac{n}{|H|}+O(t_n)\right)s_2.\]

The third sum is $O(r_n^2)$. So \[m_a=2\left(\frac{n}{|H|}+O(t_n)\right)s_2+O(r_n^2)=2\frac{n}{|H|}s_2+O(r_nt_n).\]

Recall that $M_i(a,a)=2n_a n_{-2a}+m_a$. Since $|G|$ is odd, $-2a\notin H$, so $n_a,n_{-2a}\le r_n$ and \[M_i(a,a)=2\frac{n}{|H|}s_2+O(r_nt_n)+O(r_n^2)=2\frac{n}{|H|}s_2+O(r_nt_n).\]

Similarly, if $a\in (-g_i+H)\cap G_+$, then \[M_i(a,a)=2\frac{n}{|H|}s_1+O(r_nt_n).\]

Regarding the off diagonal entries if $a\in (g_i+H)\cap G_+$ and $b\in (-g_i+H) \cap  G_+$, then
\[M_i(a,b)=M_i(b,a)=2\sqrt{n_a n_b}\left(\frac{n}{|H|}+O(t_n)\right)=2\sqrt{n_a n_b}\frac{n}{|H|}+O(r_nt_n).\]
If $a,b\in (g_i+H)\cap G_+$ and $a\neq b$, then
\[M_i(a,b)=O(r_n^2)=O(r_nt_n).\]
The same is true if $a,b\in (-g_i+H)\cap G_+$ and $a\neq b$.

Let us define the vector $v\in \mathbb{R}^{F_i\cap G_+}$ by setting
\[v(a)=\begin{cases}
\sqrt{n_a}&\text{for $a\in (g_i+H)\cap G_+$},\\
-\sqrt{n_a}&\text{for $a\in (-g_i+H)\cap G_+$}.
\end{cases}\]

Using the Cauchy-Schwarz inequality
\[\left(\sum_{a\in F_i\cap G_+}|v(a)|\right)^2\le \|v\|_2^2 |F_i\cap G_+|.\]

We have
\begin{align*}v^TM_iv&=2\sum_{a\in g_i+H}\sum_{b\in -g_i+H}-2n_an_b\frac{n}{|H|}+\sum_{a\in g_i+H}2n_a\frac{n}{|H|} s_2+\sum_{a\in -g_i+H}2n_a\frac{n}{|H|} s_1\\&\qquad+O(r_nt_n)\sum_{a\in F_i\cap G_+}\sum_{b\in F_i\cap G_+} |v(a)||v(b)|\\&=-4s_1s_2\frac{n}{|H|}+2s_1s_2\frac{n}{|H|}+2s_1s_2\frac{n}{|H|}+O(r_nt_n)\left(\sum_{a\in F_i\cap G_+}|v(a)|\right)^2\\&=O(r_nt_n)\|v\|_2^2.
\end{align*}

Thus, \[\frac{v^TM_iv}{\|v\|_2^2}=O(r_nt_n).\]

Therefore, the smallest eigenvalue of $M_i$ is at most $O(t_nr_n)$. (See \cite[Theorem 4.2.2]{horn2012matrix}.) All the other eigenvalues are at most $\Tr(M_i)=O(nr_n)$. So $\det(M_i)=O(n^{|F_i\cap G_+|-1}r_n^{|F_i\cap G_+|}t_n)$. Recall that for $a\in F_i\cap (g_i+H)$, we have
\[m_a=2\frac{n}{|H|}s_2+O(r_nt_n)\ge \frac{n}{|H|}\]
for large enough $n$. The same is true for $a\in F_i\cap (-g_i+H)$. 

Thus, for large enough $n$, we have
\begin{equation}
    \label{detMi2}\frac{\det(M_i)}{\prod_{a\in F_i\cap G_+} m_a}\le \frac{O(n^{|F_i\cap G_+|-1}r_n^{|F_i\cap G_+|}t_n)}{\left({n}/{|H|}\right)^{|F_i\cap G_+|}}\le O\left(r_n^{|F_i\cap G_+|}\frac{t_n}n\right).
\end{equation}

Combining \eqref{eqFisher} with \eqref{detMi1} and \eqref{detMi2}.
\end{proof}

\begin{lemma}\label{lemmaextralog}
For any large enough $n$ the following holds. Let $\underline{n}\in B(n,H)$. Assume that there is an $a\in G_+$ such that $(-a+H)\cap G_+=\emptyset$. Then
\[\dkl(\nu_{\underline{n}}||\mu_{\underline{n}})\ge \frac{\log n}{2n}.\]
\end{lemma}
\begin{proof}
We drop the index of $\nu$ and $\mu$. Note that

\[m_a=\sum_{b\in G} n_b n_{-a-b}=\sum_{b\in H} n_b n_{-a-b}+\sum_{b\in -a+H} n_b n_{-a-b}+\sum_{b\in G\backslash(H\cup (-a+H))} n_b n_{-a-b}.\]

If $b\in H$, then $-a-b\in -a+H$, so $n_{-a-b}=0$ by our assumption that $(-a+H)\cap G_+=\emptyset$. So the first sum above is just $0$. If $b\in -a+H$, then $n_b=0$, so the second sum is also $0$. Finally, if $b\in G\backslash(H\cup (-a+H))$, then $b\notin H$ and $-a-b\notin H$, so $n_b,n_{-a-b}\le r_n$. Thus, the third sum is at most $|G|r_n^2$.

Therefore, $\mu(a)\le \frac{|G|r_n^2}{n^2}$. Also clearly $\frac{r_n}{n}\ge\nu(a)\ge \frac{1}{n}$.
Let 
\[p=\nu(a)\quad\text{ and }\quad q=\mu(a).\]
Lemma \ref{refineq} gives us that
\begin{equation*}\dkl(\nu||\mu)\ge p\log\frac{p}q+(1-p)\log\frac{1-p}{1-q}.\end{equation*}
For fixed $p$, the right hand side as a function of $q$ is monotone decreasing on $[0,p]$. For  $n$ large enough, we have $q\le \frac{|G|r_n^2}{n^2}\le \frac{1}{n}\le p$. Thus, we can replace $q$ with $\frac{|G|r_n^2}{n^2}$ to obtain that
\[\dkl(\nu||\mu)\ge p\log\frac{pn^2}{|G|r_n^2}+(1-p)\log\frac{1-p}{1-\frac{|G|r_n^2}{n^2}}.\]
We have
\[p\log\frac{pn^2}{|G|r_n^2}\ge p\log\frac{n}{|G|r_n^2}.\]
Note that $\log(1-p)\ge (-2\log2)p$ on $[0,0.5]$. Since for large enough $n$, we have $p\in [0,0.5]$, we see that 
\[(1-p)\log\frac{1-p}{1-\frac{|G|r_n^2}{n^2}}\ge (1-p)\log(1-p)\ge -(2\log 2)(1-p)p\ge -(2\log 2)p.\]

Combining these bounds, we get
\[\dkl(\nu||\mu)\ge p\left(\log\frac{n}{|G|r_n^2}-(2\log 2)\right)\ge \frac{1}n\left(\frac{1}2\log n\right)\]
for large enough $n$.
\end{proof}

\begin{lemma}\label{lemmaMbecs2}For all $\underline{n}\in B(n,H)$, we have
\[\frac{\det(M)}{\prod_{a\in G_+} m_a}\exp(-n\dkl(\nu_{\underline{n}}||\mu_{\underline{n}}))\le O\left(\frac{\log^{|G|}n}{\sqrt{n}}\right).\]
\end{lemma}
\begin{proof}
As before let $k$ be the number of $1\le i\le h$ such that $G_+$ intersect both $g_i+H$ and $-g_i+H$. First assume that $k=0$. Since $G_+$ generates $G$, we can choose $a\in G\backslash H$. Since $k=0$, this $a$ satisfies the conditions of Lemma~\ref{lemmaextralog}, so combining this lemma with Lemma~\ref{lemmaMbecs}
\[\frac{\det(M)}{\prod_{a\in G_+} m_a}\exp(-n\dkl(\nu_{\underline{n}}||\mu_{\underline{n}}))\le O(r_n^{|G|-|H|})\exp\left(-\frac{\log n}2\right)\le O\left(\frac{\log^{|G|}n}{\sqrt{n}}\right).\]
If $k>0$, then using  Lemma~\ref{lemmaMbecs} and  Gibbs' inequality (Lemma~\ref{Gibbs}), we have
\[\frac{\det(M)}{\prod_{a\in G_+} m_a}\exp(-n\dkl(\nu_{\underline{n}}||\mu_{\underline{n}}))\le O\left(r_n^{|G|-|H|}\frac{t_n}{n}\right)\le O\left(\frac{\log^{|G|}n}{\sqrt{n}}\right).\qedhere\]
\end{proof}

Let $\underline{n}\in B(n,H)$. By the Stirling approximation \cite[ (8.22)]{rudinprinciples}, we have
\[\alpha(\underline{n})=O\left(\frac{\sqrt{n}}{\prod_{a\in G_+} \sqrt{n_a}}\right)=O\left(\frac{\sqrt{n}}{\prod_{a\in H} \sqrt{n_a}}\right)=O\left(\frac{\sqrt{n}}{ \sqrt{n/|H|-t_n}^{|H|}}\right)=O\left(\frac{1}{\sqrt{n}^{|H|-1}}\right).\]

Combining this with \eqref{eqEealpha} and Lemma \ref{lemmaMbecs2}, we get that
\[E(\underline{n})=O\left(\frac{\log^{|G|}n}{\sqrt{n}^{|H|}}\right).\]

Note that $B(n,H)$ is a subset of
\begin{multline*}\Bigg\{\underline{n}\in \mathbb{Z}^G\,:\, \left|n_a-\frac{n}{|H|}\right|\le t_n\text{ for all $0\neq a\in H$, }\\0\le n_a\le r_n\text{ for all $a\in G\backslash H$ and } n_0=n-\sum_{0\neq a\in G}n_a \Bigg\}.\end{multline*}

The size of the set above is $O\left(t_n^{|H|-1}r_n^{|G|-|H|}\right)=O\left(\sqrt{n}^{|H|-1}\log^{|G|}n\right)$.
So
\[\sum_{\underline{n}\in B(n,H)} E(\underline{n})=O\left(\sqrt{n}^{|H|-1}\log^{|G|}n\right)O\left(\frac{\log^{|G|}n}{\sqrt{n}^{|H|}}\right)=O\left(\frac{\log^{2|G|}n}{\sqrt{n}}\right)=o(1). \]

\subsection{Using Laplace's method to prove \eqref{HeG}}\label{secLaplace}

In this section, we use a variant of multivariable Laplace's method \cite[Section 6.2]{barndorff1989asymptotic} to evaluate the limit in \eqref{HeG}.

Given $(\nu(a))_{a\in G\backslash \{0\}}$, let \[\nu(0)=1-\sum_{a\in G\backslash \{0\}}\nu(a)\quad\text{ and }\quad\mu(a)=\sum_{b\in G} \nu(b)\nu(-a-b).\] 

Let $f:\mathbb{R}^{G\backslash \{0\}}\to \mathbb{R}$ be the $|G|-1$ variable function that maps $(\nu(a))_{a\in G\backslash \{0\}}$ to $ \dkl(\nu||\mu)$. 
Let $\mathbbm{1}$ be the all one vector. 

\begin{lemma}\label{lemmaderiv}
We have $f(\mathbbm{1}/|G|)=0$, the gradient of $f$ at $\mathbbm{1}/|G|$ is $0$, and the Hessian matrix of $f$ at $\mathbbm{1}/|G|$ is the $(|G|-1)\times (|G|-1)$ matrix $Q$, where the diagonal entries are $2|G|$, the offdiagonal-entries are $|G|$. 
\end{lemma}
\begin{proof}
    The proof of this lemma is a straightforward calculation, more details given in the appendix.
\end{proof}

Note that $Q$ is positive definite, and $\det Q=|G|^{|G|}$.

If $x\in \mathbb{R}^{G\backslash\{0\}}$ such that $\|x\|_\infty\le \frac{t_n}{n}$, then from Taylor's theorem
\begin{equation}\label{Taylor}
    nf\left(\frac{1}{|G|}\mathbbm{1}+x\right)=\frac{n}2 x^TQ x+O\left(\frac{t_n^3}{n^2}\right)=\frac{n}2 x^TQ x+o(1).
\end{equation}

Let $\underline{n}=\frac{n}{|G|}\mathbbm{1}$, and let $m=m_{\underline{n}}=\frac{n}{|G|}\mathbbm{1}$, and $M=M_{\underline{n}}$ as in \eqref{madef} and \eqref{Mdef}. Then $M$ is a $|G|\times |G|$ matrix, the diagonal entries of $M$ are  equal to $n^2\left(\frac{2}{|G|^2}+\frac{1}{|G|}\right)$ and the off diagonal entries are equal to $\frac{2n^2}{|G|^2}$, so \[\det(M)=3n^{2|G|}\frac{1}{|G|^{|G|}}.\] Thus,
\[\frac{\det(M)}{3\prod_{a\in G}m_a}=1.\]

By continuity, for  $\underline{n}\in B(n,G)$ and  $m=m_{\underline{n}}$, $M=M_{\underline{n}}$, we have
\begin{equation}\label{eqMcont}\frac{\det(M)}{3\prod_{a\in G}m_a}=1+o(1).\end{equation}

Using the Stirling approximation \cite[(8.22)]{rudinprinciples}, for $\underline{n}\in B(n,G)$, we have
\begin{equation}\label{eqalphaappr}
\alpha(\underline{n})=(1+o(1))\frac{\sqrt{|G|}^{|G|}}{\sqrt{2\pi n}^{|G|-1}}.    
\end{equation}

Combining \eqref{eqEealpha}, \eqref{Taylor}, \eqref{eqMcont} and \eqref{eqalphaappr}  we see $\underline{n}\in B(n,G)$
\[E(\underline{n})=(1+o(1))\frac{\sqrt{|G|}^{|G|}}{\sqrt{2\pi n}^{|G|-1}}\exp\left(-\frac{1}2 y^TQy\right),\]
with $y=\frac{P(\underline{n})-\mathbbm{1}/|G|}{\sqrt{n}}$, where $P(\underline{n})$ is the projection of $\underline{n}$ to the components indexed by $G\backslash \{0\}$.

Let \[K_n=\left\{\frac{P(\underline{n})-\mathbbm{1}/|G|}{\sqrt{n}}\,:\,\underline{n}\in B(n,G)\right\}.\]

Then
\[\sum_{\underline{n}\in B(n,G)} E(\underline{n})=(1+o(1))\sum_{y\in K_n}\frac{\sqrt{|G|}^{|G|}}{\sqrt{2\pi n}^{|G|-1}}\exp\left(-\frac{1}2 y^TQy\right). \]

The sum on the right hand side is a Riemann sum of the Gaussian integral
\[\int \frac{\sqrt{|G|}^{|G|}}{\sqrt{2\pi}^{|G|-1}}\exp\left(-\frac{1}2 y^TQy\right)=\sqrt{|G|}^{|G|}\frac{1}{\sqrt{\det{Q}}}=1.\]

So
\begin{equation}\label{eqRiemann}
    \lim_{n\to \infty} \sum_{y\in K_n}\frac{\sqrt{|G|}^{|G|}}{\sqrt{2\pi n}^{|G|-1}}\exp\left(-\frac{1}2 y^TQy\right)=1,
    \end{equation}
see the appendix for details.
Therefore,
\[\lim_{n\to\infty}\sum_{\underline{n}\in B(n,G)} E(\underline{n})=1.\]

\begin{remark}\label{remark2}
    When we proved \eqref{HeG}, we never used that $|G|$ is coprime to $6$. Thus, \eqref{HeG} is true for any finite abelian group $G$. Let $G=\mathbb{Z}/2\mathbb{Z}$. Let $\underline{n}\in \mathbb{Z}^G$ be such that $n_{0+2\mathbb{Z}}=n-1$, and $n_{1+2\mathbb{Z}}=1$. Then \eqref{elsoE} gives us that
    \begin{align*}E(\underline{n})&=n\frac{n^{-2n}}3\det\begin{pmatrix} 3n^2-6n+4&2\sqrt{n-1}\\2\sqrt{n-1}&4n-4
    \end{pmatrix} (n^2-2n+2)^{n-2}\\&=\frac{4(n-1)^3}{n^3}\left(1-\frac{2}n+\frac{2}{n^2}\right)^{n-2}=4e^{-2}+o(1).\end{align*}
    Thus, $\mathbb{E} |\Sur(\cok(A_n),G)|\ge 1+4e^{-2}+o(1)$.
    Although having bigger moments does not rule out Cohen-Lenstra limiting behaviour, in our opinion it strongly suggests that the $2$-Sylow subgroup of $\cok(A_n)$ is not Cohen-Lenstra distributed. (Note that there are distributions arbitrary close to the Cohen-Lenstra distribution in total variation with arbitrary large moments. Thus, in order to rule out the Cohen-Lenstra limiting distribution one would also need the uniform integrability of $|\Sur(\cok(A_n),G)|$, which might be possible to establish using the second moment method, but we have not pursued this question.)
\end{remark}

\appendix
\section{Appendix}
\subsection{Details of the proof of \ref{lemmaderiv}}
 We use $\partial_a$ for the partial derivative with respect to $\nu(a)$. By simple calculus
\[\partial_a f=\log(\nu(a))-\log(\nu(0))-\log(\mu(a))+\log(\mu(0))-\sum_{c\in G}\frac{\nu(c)}{\mu(c)}2(\nu(-a-c)-\nu(-c)),\]

and
\begin{align*}\partial_a\partial_a f&=\frac{1}{\nu(a)}+\frac{1}{\nu(0)}-\frac{2(\nu(-2a)-\nu(-a))}{\mu(a)}+\frac{2(\nu(-a)-\nu(0))}{\mu(0)}\\&-2\frac{\nu(-2a)}{\mu(-2a)}+2\frac{\nu(-a)}{\mu(-a)}+2\frac{\nu(-a)}{\mu(-a)}-2\frac{\nu(0)}{\mu(0)}\\&-\frac{1}{\mu(a)}2(\nu(-2a)-\nu(-a))+\frac{1}{\mu(0)}2(\nu(-a)-\nu(0))\\&+\sum_{c\in G}\frac{\nu(c)}{\mu(c)^2}4(\nu(-a-c)-\nu(-c))^2.\end{align*}

Assuming that $a\neq b$, we have
\begin{align*}\partial_b\partial_a f&=\frac{1}{\nu(0)}-\frac{2(\nu(-a-b)-\nu(-a))}{\mu(a)}+\frac{2(\nu(-b)-\nu(0))}{\mu(0)}\\&-2\frac{\nu(-a-b)}{\mu(-a-b)}+2\frac{\nu(-a)}{\mu(-a)}+2\frac{\nu(-b)}{\mu(-b)}-2\frac{\nu(0)}{\mu(0)}\\&-\frac{1}{\mu(b)}2(\nu(-a-b)-\nu(-b))+\frac{1}{\mu(0)}2(\nu(-a)-\nu(0))\\&+\sum_{c\in G}\frac{\nu(c)}{\mu(c)^2}4(\nu(-a-c)-\nu(-c))(\nu(-b-c)-\nu(-c)).
\end{align*}

\subsection{The proof of \eqref{eqRiemann}}

Let $f_n:\mathbb{R}^{G\backslash\{0\}}\to\mathbb{R}$ be defined as
\[f_n(x)=\begin{cases}\frac{\sqrt{|G|}^{|G|}}{\sqrt{2\pi}^{|G|-1}}\exp\left(-\frac{1}2 y^TQy\right)&\text{if }x\in y+\left[0,\frac{1}{\sqrt{n}}\right)^{G\backslash\{0\}}\text{ for some }y\in K_n,\\0&\text{otherwise.}\end{cases}\]

Note that 
\[\sum_{y\in K_n}\frac{\sqrt{|G|}^{|G|}}{\sqrt{2\pi n}^{|G|-1}}\exp\left(-\frac{1}2 y^TQy\right)=\int f_n.\]

The functions $f_n$ converge to $\frac{\sqrt{|G|}^{|G|}}{\sqrt{2\pi}^{|G|-1}}\exp\left(-\frac{1}2 x^TQx\right)$ pointwise. 

Since $Q$ positive definite, we have  $ y^TQy\ge c\|y\|_2^2$ for some $c>0$. Therefore, if $\|x-y\|_2^2\le |G|-1$, then
\[y^TQy\ge c\|y\|_2^2\ge c\left|\|x\|_2-\sqrt{|G|-1}\right|_+^2\]
where $|\cdot|_+$ denotes the positive part. Thus, the functions $f_n$ are dominated by the integrable function
\[\frac{\sqrt{|G|}^{|G|}}{\sqrt{2\pi}^{|G|-1}}\exp\left(-\frac{c}2\left|\|x\|_2-\sqrt{|G|-1}\right|_+^2\right).\]

Therefore, \eqref{eqRiemann} follows from the dominated convergence theorem.

\bibliography{references}
\bibliographystyle{plain}

\bigskip

\noindent Andr\'as M\'esz\'aros, \\
HUN-REN Alfr\'ed R\'enyi Institute of Mathematics,\\ {\tt meszaros@renyi.hu}

\end{document}